\documentclass[11pt]{amsart}
\usepackage[a4paper,margin=2.5cm]{geometry}
\usepackage{hyperref,amsmath,amssymb,color,tikz-cd}
\usepackage{diagbox, arydshln}
\usepackage{nicematrix}

\theoremstyle{plain}
\newtheorem{theorem}{Theorem}[section]
\newtheorem{proposition}[theorem]{Proposition}
\newtheorem{lemma}[theorem]{Lemma}
\newtheorem{corollary}[theorem]{Corollary}
\numberwithin{equation}{section}

\theoremstyle{definition}
\newtheorem{definition}[theorem]{Definition}
\newtheorem{remark}[theorem]{Remark}

\newcommand{\Z}{\mathbb{Z}}

\newcommand{\R}{\mathbb{R}}

\newcommand{\cC}{\mathcal{C}}
\newcommand{\cR}{\mathcal{R}}
\newcommand{\cA}{\mathcal{A}}
\newcommand{\RP}{\mathbb{R}P}
\newcommand{\RZ}{\mathbb{R}\mathcal{Z}}
\newcommand{\cZ}{\mathcal{Z}}

\DeclareMathOperator{\GL}{GL}
\DeclareMathOperator{\Lk}{Lk}
\DeclareMathOperator{\row}{row}
\DeclareMathOperator{\Bier}{Bier}
\DeclareMathOperator{\Aut}{Aut}
\DeclareMathOperator{\ind}{\mathbf{1}}

\begin{document}
\title[On small covers over Bier spheres]{On small covers over Bier spheres}

\author[S. Choi]{Suyoung Choi}
\address{Department of mathematics, Ajou University, 206, World cup-ro, Yeongtong-gu, Suwon 16499,  Republic of Korea}
\email{schoi@ajou.ac.kr}

\author[Y. Yoon]{Younghan Yoon}
\address{Research Institute of Basic Sciences, Ajou University, 206, World cup-ro, Yeongtong-gu, Suwon 16499, Republic of Korea}
\email{younghan300@ajou.ac.kr}

\author[S. Yu]{Seonghyeon Yu}
\address{Department of mathematics, Ajou University, 206, World cup-ro, Yeongtong-gu, Suwon 16499,  Republic of Korea}
\email{yoosh0319@ajou.ac.kr}

\date{\today}


\thanks{This work was supported by the National Research Foundation of Korea Grant funded by the Korean Government (RS-2025-00521982).}

\begin{abstract}
    The Bier sphere of a simplicial complex~$K$ is defined as the deleted join of~$K$ and its combinatorial Alexander dual.
    We focus on the class of Bier spheres of the skeleta of a simplex.
    Since these Bier spheres are known to be polytopal, they give rise to small covers.
    We classify small covers over these Bier spheres up to Davis--Januszkiewicz equivalence.
     
    As applications, for all~$m \geq 4$, we determine the homeomorphism types of small covers over the Bier spheres of the~$0$-skeleton and the~$(m-3)$-skeleton of an~$(m-1)$-simplex.
    For the remaining cases~$0<r<m-3$, we compute their rational Betti numbers.
\end{abstract}
\maketitle

\section{Introduction} \label{sec:intro}
Let~$P^n$ be an~$n$-dimensional simple convex polytope.
A \emph{small cover} over~$P^n$, introduced by Davis and Januszkiewicz~\cite{Davis-Januszkiewicz1991}, is an~$n$-dimensional closed smooth manifold~$M$ equipped with a locally standard~$\Z_2^n$-action together with a projection~$\pi\colon M \to P^n$ such that the orbit space~$M/\Z_2^n$ is identified with~$P^n$.
Two small covers~$\pi_1 \colon M_1 \to P^n$ and~$\pi_2 \colon M_2 \to P^n$ are said to be \emph{Davis--Januszkiewicz equivalent} (or, simply, \emph{D--J equivalent}) if there exist a group automorphism~$\theta\in \Aut(\Z_2^n)$ and a homeomorphism~$f \colon M_1 \to M_2$ such that~$f(t\cdot x)=\theta(t)\cdot f(x)$ for all~$t\in \Z_2^n$ and~$x\in M_1,$ and the diagram
$$
    \begin{tikzcd}
        M_1 \arrow[rr,"f"] \arrow[dr,"\pi_1"'] & & M_2 \arrow[dl,"\pi_2"] \\
        & P^n &
    \end{tikzcd}
$$
commutes.
The problem of identifying all small covers has been studied for various specific classes of simple polytopes~\cite{Garrison-Scott2003, Cai-Chen-Lu2007, Choi2008, Choi-Masuda-Suh2010Quasitoric, Choi2010, Wang-Chen2012, Hasui2015, Choi-Park2019SmallCovers}.

Assume that~$m$ is a positive integer. 
Let~$[m] = \{1, \ldots, m\}$ be the standard vertex set and~$[\overline{m}] = \{\overline{1}, \ldots, \overline{m}\}$ its disjoint copy.
A \emph{Bier sphere}, introduced by Bier in~\cite{Bier1992}, arises from a simplicial complex and its Alexander dual as follows.
For a simplicial complex~$K\subsetneq 2^{[m]}$, its \emph{Alexander dual}~$\widehat{K}$ is the simplicial complex on~$[\overline{m}]$ consisting of the faces~$\overline{\sigma}$ such that~$[m]\setminus\sigma$ is not a face of~$K$.
By taking the deleted join of~$K$ and~$\widehat{K}$, we obtain the Bier sphere~$\Bier K$ of~$K$, which is an~$(m-2)$-dimensional simplicial sphere on (up to)~$2m$ vertices.
One of the basic structural features of Bier spheres is that every Bier sphere is shellable, although it need not be polytopal~\cite{Bjorner-Paffenholz-Sjostrand-Ziegler2005}.
More recently, it was shown that every Bier sphere carries a canonical complete nonsingular fan~\cite{Ivan-Marinko-Rade2025}.
Beyond these results, various combinatorial and topological properties of Bier spheres have been studied; see, for example,~\cite{Murai2011,Heudtlass-Katthan2012,Jevtic-Timotijevic-Zivaljevic2021,Limonchenko-Sergeev2024,Limonchenko-Vavpeti2024,Choi-Yoon-Yu2026}.

We denote by~$\Delta_r^{m-1}$ the~$r$th skeleton of the simplex with~$m$ vertices.
Among Bier spheres, the family associated with~$\Delta_r^{m-1}$ occupies a particularly distinguished position.
On the one hand, these examples are sufficiently symmetric and rigid to allow explicit combinatorial and topological computations.  
On the other hand, they are rich enough to exhibit several phenomena that also occur for general Bier spheres.  
In this sense, they serve as a useful testing ground for the broader theory.
For instance, the Betti numbers of full subcomplexes of~$\Bier\Delta_r^{m-1}$ were computed in~\cite{Heudtlass-Katthan2012}, and this direction was later extended to homotopy type descriptions for full subcomplexes of arbitrary Bier spheres in~\cite{Choi-Yoon-Yu2026}.  
Moreover, this family appears naturally in the classical proof of the Van Kampen--Flores theorem, where the combinatorics of the skeleta of a simplex plays a central role; see~\cite{Matousek2003book}.

In this paper, we focus on the Bier spheres~$\Bier\Delta_{r}^{m-1}$.
Since these Bier spheres are known to be polytopal~\cite{Jevtic-Timotijevic-Zivaljevic2021}, we study small covers over them and classify these small covers up to D--J equivalence.

Section~\ref{sec3} presents Theorem~\ref{thm:r=0_sc}, which gives a more detailed version of Theorem~\ref{thm:main1} by providing explicit representatives, while Section~\ref{sec:proof} contains its proof.
\begin{theorem}\label{thm:main1}
    Let~$m\geq 4$ and~$0\leq r\leq m-3$.
    The number of small covers over~$\Bier\Delta_{r}^{m-1}$ up to D--J equivalence is given by
    $$
        \begin{cases}
            1+m \cdot 2^{m-2}, & \text{if } r = 0 \text{ or } r=m-3,\\
            1+m, & \text{if } 0 < r < m-3.
        \end{cases}
    $$
\end{theorem}

In Section~\ref{sec4}, we study the topology of small covers over~$\Bier\Delta_{r}^{m-1}$.
First, when~$r=0$, Corollary~\ref{thm:main2} weakens the classification in Theorem~\ref{thm:main1} to a classification up to homeomorphism.
We denote by~$\RP^n$ the~$n$-dimensional real projective space.

\begin{corollary}\label{thm:main2}
    For~$m\geq 4$ and~$r\in\{0,m-3\}$, there are at most three small covers over~$\Bier\Delta_{r}^{m-1}$ up to homeomorphism.
    More precisely, each of them is homeomorphic to one of the following manifolds
    $$
        \#^{m+1}\RP^{m-1}, \quad (\#^{m-1}\RP^{m-1})\#(\RP^{m-2}\times \RP^1), \quad \text{or} \quad (\#^{m-1}\RP^{m-1})\#\RP(\gamma\oplus \R^{m-2}),
    $$
    where~$\gamma$ denotes the tautological line bundle over~$\RP^1$.
\end{corollary}
When~$0<r<m-3$, the problem of reducing the classification to one up to homeomorphism remains open.
Instead, we study these small covers through their rational Betti numbers. One exceptional case was already computed in~\cite{Choi-Yoon-Yu2026} and is recalled in Proposition~\ref{prop:Bettisk}.
For the remaining cases, Corollary~\ref{thm:betti_r_positive} gives an explicit formula.
Combining these results, we obtain Corollary~\ref{cor:main}, which determines the rational Betti numbers of all small covers over~$\Bier\Delta_r^{m-1}$.
We denote by~$\beta_{i}(X)$ the $i$th rational Betti number of a topological space~$X$.
\begin{corollary}\label{cor:main}
    Let~$m\geq 5$ and~$0<r\leq \left\lfloor \frac{m-3}{2}\right\rfloor$. 
    Among the D--J equivalence classes of small covers over~$\Bier\Delta_r^{m-1}$, exactly one has rational Betti numbers given, for each~$k\geq0$, by
$$
\beta_{2k}(M) = \dbinom{m}{2k} \cdot \ind_{\{0,1,\ldots,r+1\}}(2k),
$$
and, for each~$k \geq 1$, by
$$
\beta_{2k-1}(M) = \dbinom{m}{2k} \cdot \ind_{\{m-1-r,\ldots,m\}}(2k).
$$
For small covers~$M$ in all the other D--J equivalence classes, the rational Betti numbers are given, for each~$k\geq 0$, by
$$
\beta_{2k}(M) = \dbinom{m-1}{2k} \cdot \ind_{\{0,1,\ldots,r+1\}}(2k),
$$
and, for each~$k\geq 1$, by
$$
\beta_{2k-1}(M) = \dbinom{m-1}{2k} \cdot \ind_{\{m-r-1,\ldots,m-1\}}(2k) + \dbinom{m-2}{r+1} \cdot \ind_{\{m-r-1\}}(2k) + \dbinom{m-2}{r} \cdot \ind_{\{r+2\}}(2k).
$$
\end{corollary}

\section{Preliminaries}
Let~$S$ denote a finite set of vertices. 
An abstract \emph{simplicial complex}~$\Gamma$ on the vertex set~$S$ is a collection of subsets of~$S$ such that every subset of a set in~$\Gamma$ must also belong to~$\Gamma$.
An element of~$\Gamma$ is called a \emph{face} or a \emph{simplex}, and the \emph{dimension} of a face~$\sigma$ is given by~$\left\vert \sigma \right\vert - 1$. 
A \emph{facet} is defined as an inclusion-maximal face in~$\Gamma$.
The \emph{dimension} of~$\Gamma$, denoted~$\dim \Gamma$, is the maximum dimension among all its facets.
The complex~$\Gamma$ is said to be \emph{pure} when every facet possesses the same dimension.
A \emph{subcomplex} is a subcollection of~$\Gamma$ that is itself a simplicial complex.
For any~$I \subset S$, the \emph{full subcomplex}~$\Gamma_I$ restricted to~$I$ is defined as~$\Gamma_I= \{\sigma \in \Gamma \colon \sigma \subseteq I\}$.

Let~$\Gamma$ be a pure simplicial complex on~$[m]$.
The \emph{real moment-angle complex}~$\R\cZ_{\Gamma}$ associated with~$\Gamma$ is defined by
$$
\R\cZ_{\Gamma} = \bigcup_{\sigma \in \Gamma}\bigl\{ (x_1,\ldots,x_m) \in (D^1)^m \mid x_i \in S^0 \text{ for } i \notin \sigma \bigr\},
$$
where $D^1 = [-1,1]$ and $S^0 = \{-1,1\}$ is its boundary.

Let~$\lambda \colon \Z_2^m \longrightarrow \Z_2^n$ be a linear map with~$n \leq m$.
We often represent~$\lambda$ by an~$n \times m$ $\Z_2$-matrix
$$
\Lambda = \big( \lambda(e_1) \ \cdots \ \lambda(e_m) \big),
$$
where~$e_i$ denotes the~$i$th standard basis vector of~$\Z_2^m$.
We say that~$\lambda$, or equivalently~$\Lambda$, satisfies the
\emph{non-singularity condition} over~$\Gamma$ if, for every simplex~$\{i_1,\ldots,i_k\}$ of~$\Gamma$, the set
$$
\{\lambda(e_{i_1}),\ldots,\lambda(e_{i_k})\}
$$
is linearly independent in~$\Z_2^n$.

The canonical sign action of~$\Z_2^m$ on $(D^1)^m$ restricts to an action on~$\RZ_\Gamma$, and hence induces an action of~$\ker\lambda$ on~$\RZ_\Gamma$.
By~\cite{Choi-Kaji-Theriault2017}, this action is free if and only if~$\lambda$ satisfies the non-singularity condition over~$\Gamma$.
When this condition holds, $\lambda$ is called a \emph{mod~$2$ characteristic map} over~$\Gamma$, and~$\Lambda$ is called a \emph{mod~$2$ characteristic matrix} over~$\Gamma$.
Since we only consider the mod~$2$ setting in this paper, we shall simply call them a \emph{characteristic map} and a \emph{characteristic matrix}, respectively.
The associated \emph{real toric space} is defined by
$$
    M(\Gamma,\lambda) := \RZ_\Gamma/\ker\lambda.
$$

Let us consider the case where~$\Gamma$ is polytopal, that is, $\Gamma$ is the dual complex of a simple polytope.
In this case, for a characteristic map~$\lambda$, the associated real toric space~$M(\Gamma,\lambda)$ is called a \emph{small cover}~\cite{Davis-Januszkiewicz1991}.
For~$n \times m$ characteristic matrices~$\Lambda_1$ and~$\Lambda_2$ over~$\Gamma$, we say that~$\Lambda_1$ and~$\Lambda_2$ are \emph{Davis--Januszkiewicz equivalent} (or, simply, \emph{D--J equivalent}) if there exists a matrix~$P \in \GL(n,\Z_2)$ such that~$\Lambda_2 = P\Lambda_1$.
Accordingly, the corresponding small covers~$M(\Gamma,\Lambda_1)$ and~$M(\Gamma,\Lambda_2)$ are also said to be D--J equivalent.
This agrees, in essence, with the definition of D--J equivalence given in Section~\ref{sec:intro}.

Now, we introduce notions of \emph{Bier spheres}, originally defined in~\cite{Bier1992}.
For a given set of integers~$I$, let~$\overline{I} = \{\overline{i} \mid i \in I\}$ denote its disjoint isomorphic copy. 
Suppose~$K$ is a simplicial complex on a finite set~$S$. 
The (combinatorial) \emph{Alexander dual} of~$K$, denoted by~$\widehat{K}$, is a simplicial complex on~$\overline{S}$ defined by
$$
    \widehat{K} = \{\overline{\sigma} \subseteq \overline{S} \mid S \setminus \sigma \notin K\}.
$$
Under the canonical identification~$\overline{\overline{S}} = S$, the dual of~$\widehat{K}$ is exactly~$K$ itself.
For any two simplicial complexes~$K_1$ on~$S$ and~$K_2$ on~$\overline{S}$, their \emph{deleted join}, written as~$K_1 \ast_{\Delta} K_2$, is given by
$$
    K_1 \ast_{\Delta} K_2 = \{\sigma \cup \overline{\tau} \mid \sigma \in K_1, \overline{\tau} \in K_2, \sigma \cap \tau = \emptyset\}.
$$

For a simplicial complex~$K$ on~$[m] = \{1,\ldots,m\}$, its Bier sphere~$\Bier K$ is a simplicial complex on the disjoint union~$[m] \cup [\overline{m}]$, constructed as the deleted join
$$
    \Bier K = K \ast_{\Delta} \widehat{K}.
$$
Every Bier sphere is an~$(m-2)$-dimensional simplicial sphere.
However, Bier spheres are not necessarily polytopal~\cite{Bjorner-Paffenholz-Sjostrand-Ziegler2005}.
Nevertheless, some Bier spheres are known to be polytopal.
Let~$\Delta_{r}^{m-1}$ be the~$r$-skeleton of the~$(m-1)$-simplex on~$[m]$, namely
$$
    \Delta_{r}^{m-1}=\{\sigma \subseteq [m] \mid |\sigma| \leq r+1\}.
$$
The associated Bier sphere~$\Bier\Delta_{r}^{m-1}$ is known to be polytopal~\cite{Jevtic-Timotijevic-Zivaljevic2021}.

For each integer~$m \geq 2$, define an~$(m-1)\times 2m$ matrix~$\Lambda_m$ over~$\Z_2$ by
\begin{equation}\label{eq:matrix} 
    \Lambda_m := 
        \begin{pNiceArray}{cccccccc}[first-row,margin,code-for-first-row=\scriptstyle] 
            1 & \cdots & m-1 & m & \overline{1} & \cdots & \overline{m-1} & \overline{m}\\
            1 & \cdots & 0 & 1 & 1 & \cdots & 0 & 1\\
            \vdots & \ddots & \vdots & 1 & \vdots & \ddots & \vdots & 1\\
            0 & \cdots & 1 & 1 & 0 & \cdots & 1 & 1
        \end{pNiceArray} 
\end{equation}
where the first row records the labels of the vertices.
For a pure~$(m-2)$-dimensional simplicial complex~$\Gamma$ on~$[m]\cup[\overline{m}]$, the matrix~$\Lambda_m$ is a mod~$2$ characteristic matrix over~$\Gamma$ if and only if, for each~$1 \leq i \leq m$, the vertices~$i$ and~$\overline{i}$ are not adjacent in~$\Gamma$.
In particular, $\Lambda_m$ is a mod~$2$ characteristic matrix over any Bier sphere. 
See~\cite{Limonchenko-Sergeev2024,Ivan-Marinko-Rade2025} for further details.
This gives a natural first example of a mod~$2$ characteristic matrix over an arbitrary Bier sphere.
Thus, every Bier sphere admits at least one mod~$2$ characteristic matrix.
In this paper, we focus on the class of Bier spheres~$\Bier\Delta_r^{m-1}$ and determine all characteristic matrices over these spheres up to D--J equivalence.

\section{The characteristic matrices over~$\Bier\Delta_{r}^{m-1}$}\label{sec3}

In this paper, we assume that~$m \geq 4$, unless otherwise stated. 
Let~$\Delta^{m-1}$ be the~$(m-1)$-simplex on the vertex set~$[m]=\{1,\ldots,m\}$.
For~$0 \leq r \leq m-3$, let~$\Delta_{r}^{m-1}$ denote the~$r$-skeleton of~$\Delta^{m-1}$.
Note that the Alexander dual of~$\Delta_{r}^{m-1}$ coincides with~$\Delta_{m-r-3}^{m-1}$.
Therefore, in order to study the Bier spheres of all skeleta of~$\Delta^{m-1}$, it suffices to consider~$\Bier\Delta_{r}^{m-1}$ for
$$
    0 \leq r \leq \left\lfloor \frac{m-3}{2} \right\rfloor.
$$
Unless otherwise stated, we shall make this assumption on~$r$ throughout Sections~\ref{sec3} and~\ref{sec:proof}.

Now, let 
$$
    \cR_r := \{1,\ldots,r+1,\overline{r+2},\ldots,\overline{m-1}\}.
$$
Then~$\cR_r$ is a facet of~$\Bier\Delta_{r}^{m-1}$. 
Hence, the characteristic matrix on~$\Bier\Delta_{r}^{m-1}$ can, up to a change of basis, be written in the form
\begin{equation}\label{eq:mat_form}
    \Lambda(\cA) = \begin{pNiceArray}{c:cc}[first-row,margin,code-for-first-row=\scriptstyle]
        \begin{array}{*{6}{>{\scriptstyle}c}} 1 & \cdots & r+1 & \overline{r+2} & \cdots & \overline{m-1} \end{array} & \begin{array}{*{7}{>{\scriptstyle}c}} \overline{1} & \cdots & \overline{r+1} & r+2 & \cdots & m & \overline{m} \end{array} \\
        &\\
        I_{m-1} & \cA \\
        &
    \end{pNiceArray}
\end{equation}
where~$I_{m-1}$ denotes the identity matrix of order~$m-1$, and~$\cA$ is an~$(m-1) \times (m+1)$ matrix.
Therefore, the classification of small covers over~$\Bier\Delta_{r}^{m-1}$ up to D--J equivalence reduces to determining all such matrices~$\cA$ satisfying the non-singularity condition.

We denote by~$J_{k,\ell}$ the~$(k \times \ell)$ matrix over~$\Z_2$ whose entries are all equal to~$1$.
If we permute the columns of~$\Lambda_m$ in~\eqref{eq:matrix} to obtain the form in~\eqref{eq:mat_form}, then~$\cA$ becomes precisely the matrix
$$
    A_0 := \begin{pNiceArray}{c:c}[first-row,margin,code-for-first-row=\scriptstyle]
        \begin{array}{*{6}{>{\scriptstyle}c}} \overline{1} & \cdots & \overline{r+1} & r+2 & \cdots & m-1 \end{array}   & \begin{array}{*{2}{>{\scriptstyle}c}} m & \overline{m} \end{array}\\
        &\\
        I_{m-1}& J_{m-1,2} \\
        &
    \end{pNiceArray}.
$$

Throughout this paper, for any positive integer~$\ell$, we regard each element of~$\Z_2^\ell$ as either a row vector or a column vector, depending on the context.
\begin{definition}\label{def:mat}
    Assume that~$m \geq 4$.
    Let~$v = (v_{r+2},\ldots,v_{m-1})\in \Z_2^{m-r-2}$.
    \begin{enumerate}
        \item 
            \begin{equation*}
                A_{v} = \begin{pNiceArray}{c:c:c}[first-row,margin,code-for-first-row=\scriptstyle]
                    \begin{array}{*{6}{>{\scriptstyle}c}} \overline{1} & \cdots & \overline{r+1}&r+2 & \cdots & m-1 \end{array} & m & \overline{m}\\
                    &&\\
                    J_{m-1} - I_{m-1} & \begin{array}{c} J_{r+1,1}\\ \\ \hdashline \\ v^T \end{array} & J_{m-1,1}\\
                    &&
                \end{pNiceArray}.
            \end{equation*}
        \item 
            Let~$M^{v}_j$ be the~$(m-r-2) \times j$ matrix over~$\Z_2$ whose columns are all equal to~$v$.
            For an element~$s \in \{1,\ldots,r+1\}$, let~$B_v^s$ be the matrix obtained from 
            \begin{equation}\label{r=0_case3}
                \begin{pNiceArray}{c:c:c:c}[first-row,margin,code-for-first-row=\scriptstyle]
                    \begin{array}{*{3}{>{\scriptstyle}c}} \overline{1} & \cdots & \overline{r+1} \end{array} & \begin{array}{*{3}{>{\scriptstyle}c}} r+2 & \cdots & m-1 \end{array} & m & \overline{m}\\
                    &&&\\
                    I_{r+1} & 0  & J_{r+1,1} & J_{r+1,1}\\
                    &&&\\ \hdashline 
                    && v_{r+2}+1 & \\
                    M^v_{r+1} & I_{m-r-2}+M^v_{m-r-2}& \vdots &J_{m-r-2,1} \\
                    && v_{m-1}+1&
                \end{pNiceArray}
            \end{equation}
            by replacing its~$s$th row with
            \begin{equation}\label{eq:B_s0}
                \begin{pNiceArray}{c:cc}[first-row,margin,code-for-first-row=\scriptstyle] 
                    \begin{array}{*{6}{>{\scriptstyle}c}} \overline{1} & \cdots & \overline{r+1} & r+2 & \cdots & m-1 \end{array}   & m & \overline{m}\\
                    J_{1,m-1}& 1 & 0
                \end{pNiceArray} \in \Z_2^{m+1}.
            \end{equation}
        \item 
            For an element~$t \in \{r+2,\ldots,m-1\}$, let~$C_v^t$ be the matrix obtained from~$A_0$ by adding
            $$
                \begin{array}{@{}r@{\,}}
                    r+1 \left\{ \begin{array}{c} \vphantom{1+v_s}\\ \vphantom{\vdots}\\ \vphantom{1+v_s} \end{array} \right. \\
                    \vphantom{v_{r+2}}\\
                    \vphantom{\vdots}\\
                    \vphantom{v_{m-1}}
                \end{array}
                \left(
                    \begin{array}{c}
                        1+v_t\\
                        \vdots\\
                        1+v_t\\
                        v_{r+2}\\
                        \vdots\\
                        v_{m-1}
                    \end{array}
                \right) \in \Z^{m-1}_2
            $$
            to its~$t$th column and
            $$
                \begin{pNiceArray}{c:cc}[first-row,margin,code-for-first-row=\scriptstyle] 
                    \begin{array}{*{6}{>{\scriptstyle}c}} \overline{1} & \cdots & \overline{r+1} & r+2 & \cdots & m-1 \end{array}   & m & \overline{m}\\
                    J_{1,m-1}& 1 & 0
                \end{pNiceArray} \in \Z^{m+1}_2
            $$
            to its~$t$th row.
    \end{enumerate}
\end{definition}

\begin{remark}\label{remark:main}
    For~$r=0$, all matrices in Definition~\ref{def:mat} give characteristic matrices over~$\Bier\Delta^{m-1}_0$.
    If~$r>0$, then the corresponding matrix~$\Lambda(\cA)$ in~\eqref{eq:mat_form} satisfies the non-singularity condition for each of the following cases:
    \begin{enumerate}
        \item $\cA = A_{(1,\ldots,1)}$;
        \item $\cA = B^{s}_{(0,\ldots,0)}$ for $s \in \{1,\ldots,r+1\}$; and
        \item $\cA = C^t_v$ for $t\in \{r+2,\ldots,m-1\}$, where $v=(v_{r+2},\ldots,v_{m-1})$ satisfies $v_t=1$ and $v_i=0$ for all $i\neq t$.
    \end{enumerate}
\end{remark}

Theorem~\ref{thm:r=0_sc} gives the complete list of matrices that actually occur among those defined above and counts the resulting D--J equivalence classes.
The proof is given in Section~\ref{sec:proof}.

\begin{theorem}\label{thm:r=0_sc}
    Let~$m \geq 4$.
    The number of mod~$2$ characteristic matrices over~$\Bier\Delta_{r}^{m-1}$ up to D--J equivalence is given by
    $$
        \begin{cases}
            1+m \cdot 2^{m-2}, & \text{if } r = 0, \\
            1+m, & \text{if } 0 < r \leq \left\lfloor \frac{m-3}{2} \right\rfloor.
        \end{cases}
    $$
    Moreover, every such matrix is obtained from~\eqref{eq:mat_form} by taking~$\cA$ to be one of the following matrices.
    When~$r = 0$,
    \begin{enumerate}
        \item $\cA = A_0$;
        \item $\cA = A_v$ for $v \in \Z_2^{m-2}$;
        \item $\cA = B^1_v$ for $v \in \Z_2^{m-2}$;
        \item $\cA = C^t_v$ for $t\in\{r+2,\ldots,m-1\}$ and $v \in \Z_2^{m-2}$.
    \end{enumerate}
    When~$0 < r \leq \left\lfloor \frac{m-3}{2} \right\rfloor$,
    \begin{enumerate}
        \item $\cA = A_0$;
        \item $\cA = A_v$, where~$v = \left(1,\ldots,1\right)$;
        \item $\cA = B^s_v$ for $s\in \{1,\ldots,r+1\}$, where~$v =\left(0,\ldots,0\right)$;
        \item $\cA = C^t_v$ for $t\in\{r+2,\ldots,m-1\}$, where~$v=(v_{r+2},\ldots,v_{m-1})$ satisfies~$v_t=1$ and~$v_i=0$ for all~$i\neq t$.
    \end{enumerate}
\end{theorem}

Note that Theorem~\ref{thm:r=0_sc} is stated for~$m\geq4$.
For~$m=3$, $\Bier\Delta_0^2$ is the boundary complex of a hexagon, and the classification of small covers over it up to D--J equivalence is well known~\cite{Davis-Januszkiewicz1991}.
Moreover, the following characteristic matrix
$$
    \begin{pNiceArray}{cccccc}[first-row,margin,code-for-first-row=\scriptstyle]
        1 &\overline{1}&2 &\overline{2}&3 &\overline{3}\\
        1& 0 & 1& 0 &1& 0 \\
        0 &1& 0 &1& 0 &1
    \end{pNiceArray}
$$
also occurs, but it is not covered by the families listed in Theorem~\ref{thm:r=0_sc}.

\section{Proof of Theorem~\ref{thm:r=0_sc}}\label{sec:proof}

In this section, we prove Theorem~\ref{thm:r=0_sc}.
Recall~$\cR_r= \{1,\ldots,r+1,\overline{r+2},\ldots,\overline{m-1}\} \subset [m] \cup [\overline{m}]$.
For each~$r+2 \leq k \leq m-1$, the symbol~$\overline{k}$ in the row set~$\cR_r$ represents the row of~$\cA$ originally indexed by~$k$.
We denote the column set of~$\cA$ by
$$
    \cC_r= \{\overline{1},\ldots,\overline{r+1},r+2,\ldots,m-1,m,\overline m\}.
$$
In this section, we denote the entries of a matrix~$\cA$ by
$$
    \cA=(a_{i,j})_{i\in\cR_r,  j\in\cC_r}.
$$
Let~$X \subset \cR_r$ and~$Y \subset \cC_r$.
Note that~$\cR_r$ is a facet of~$\Bier \Delta_{r}^{m-1}$.
A pair~$(X,Y)$ is called a \emph{replacement pair} for~$\cR_r$ if~$(\cR_r \setminus X)\cup Y$ is also a facet of~$\Bier \Delta_{r}^{m-1}$.

For a matrix~$M$, let~$X$ and~$Y$ be subsets of the row and column sets of~$M$, respectively.
We denote by~$[M]_{X,Y}$ the minor of~$M$ determined by~$X$ and~$Y$.

\begin{proposition}\label{prop:minor}
    A pair~$(X,Y)$ is a replacement pair for~$\cR_r$ if and only if
    \begin{enumerate}
        \item $|X \cap [m]| = |Y \cap [m]|$;
        \item $|X \cap [\overline{m}]| = |Y \cap [\overline{m}]|$;
        \item if~$i \in Y \setminus \{m,\overline{m}\}$, then~$\overline{i} \in X$;
        \item $\{m,\overline{m}\} \not\subset Y$.
    \end{enumerate}
    Moreover, for every replacement pair~$(X,Y)$ for~$\cR_r$,
    \begin{equation}\label{eq:minor_cal}
        [\cA]_{X,Y}=1.
    \end{equation}
\end{proposition}
\begin{proof}
    A subset~$F \subset [m]\cup[\overline{m}]$ is a facet of~$\Bier\Delta_{r}^{m-1}$ if and only if
    $$
        |F \cap [m]| = r+1,\quad |F \cap [\overline{m}]| = m-r-2,\quad  \text{and} \quad \{i,\overline{i}\}\not\subset F \text{ for all }1 \leq i \leq m.
    $$
    Applying the above characterization of facets to~$F=(\cR_r\setminus X)\cup Y$, the first two conditions become~(1) and~(2), respectively, while the last condition becomes~(3) and~(4) together.
    Hence, the first statement follows immediately.

    Let~$\cR_r' = (\cR_r \setminus X)\cup Y$.
    Then,
    $$
        [\cA]_{X,Y} = [\Lambda(\cA)]_{\cR_r,\cR_r'} = 1.
    $$
\end{proof}

\begin{table}
    \centering
    \begin{tabular}{c p{0.38\textwidth} p{0.42\textwidth}}
        \hline
         & Replacement pair $(X,Y)$ & Consequence of $[\cA]_{X,Y}=1$ \\ \hline
        (1) & $(\{s\},\{m\})$ &$a_{s,m}=1$ \\
        (2) & $(\{\overline{t}\},\{\overline{m}\})$ & $a_{\overline{t},\overline{m}}=1$ \\
        (3) & $(\{s,\overline{t}\},\{t,\overline{s}\})$ & $a_{s,\overline{s}}a_{\overline{t},t}+a_{s,t}a_{\overline{t},\overline{s}}=1$ \\
        (4) & $(\{s,\overline{t}\},\{m,\overline{s}\})$ & $a_{s,\overline{s}}a_{\overline{t},m}+a_{\overline{t},\overline{s}}=1$ \\
        (5) & $(\{s,\overline{t}\},\{t,\overline{m}\})$ & $a_{s,t}+a_{s,\overline{m}}a_{\overline{t},t}=1$ \\
        (6) & $(\{s,\overline{t},\overline{t'}\}, \{t,\overline{m},\overline{s}\})$ & $\begin{aligned}[t] a_{\overline{t'},\overline{s}} &+a_{s,\overline{s}}a_{\overline{t'},t}+a_{s,\overline{m}}a_{\overline{t},\overline{s}}a_{\overline{t'},t}=0 \end{aligned}$ \\ \hline
    \end{tabular}
    \caption{Relations obtained from~\eqref{eq:minor_cal}.} \label{tab:minor-cal-relations}
\end{table}

The consequences of~\eqref{eq:minor_cal} for several small replacement pairs are summarized in Table~\ref{tab:minor-cal-relations}.
Throughout the table, the indices satisfy~$1\leq s\leq r+1$, $r+2\leq t, t'\leq m-1$ and~$t \neq t'$.
Relations~(1),~(2), and~(3) of Table~\ref{tab:minor-cal-relations} follow directly from~\eqref{eq:minor_cal}.
Relations~(4) and~(5) follow from~\eqref{eq:minor_cal} after using relations~(1) and~(2).
Finally, relation~(6) follows from~\eqref{eq:minor_cal} after using relation~(5), and hence indirectly from the preceding rows.

\begin{remark}
    When~$r=0$, the only possible choice of~$s$ is~$s=1$. 
    Thus, as~$t$ and~$t'$ vary, Table~\ref{tab:minor-cal-relations} contains all relations arising from~\eqref{eq:minor_cal}.
\end{remark}

In the same way as in relation~(6) of Table~\ref{tab:minor-cal-relations}, when~$r>0$, we have additional three-element replacement pairs.
For distinct~$1\leq s,s'\leq r+1$ and~$r+2\leq t\leq m-1$, taking~$(X,Y)=(\{s,s',\overline{t}\}, \{t,m,\overline{s}\})$ in~\eqref{eq:minor_cal}, we obtain
\begin{equation}\label{eq:minor-3b}
    a_{s,t} +a_{s,\overline{s'}}a_{\overline{t},t} +a_{\overline{t},m}a_{s',t}a_{s,\overline{s'}} = 0.
\end{equation}
The proof of Theorem~\ref{thm:r=0_sc} is based on the following three lemmas.

\begin{lemma}\label{lemma:main1}
    If~$\Lambda(\cA)$ satisfies the non-singularity condition and~$a_{s,\overline{s}}=0$ for some~$1 \leq s \leq r+1$, then~$\cA$ coincides with one of the following matrices.
    $$
        \cA = \begin{cases}
            A_v \quad \text{for } v \in \Z_2^{m-2}, &  \text{if } r=0,\\
            A_{(1,\ldots,1)} & \text{if } r>0.
        \end{cases}
    $$
\end{lemma}
\begin{proof}
    Take~$s_0 \in \{1,\ldots,r+1\}$ with~$a_{s_0,\overline{s_0}}=0$.
    After putting~$s=s_0$ in Table~\ref{tab:minor-cal-relations},
    the resulting equations, simplified successively, are precisely the following identities
    $$
        \begin{cases}
            a_{s_0,m} = a_{\overline{t},\overline{m}} = a_{s_0,t} = a_{\overline{t},\overline{s_0}} = a_{s_0,\overline{m}} =1, \quad a_{\overline{t},t}=0 &  \text{if } r+2 \leq t \leq m-1,\\
            a_{\overline{t'},t}=1& \text{if } r+2 \leq t,t' \leq m-1 \text{ with } t \neq t'.
        \end{cases}
    $$
    
    When~$r= 0$, we may write
    $$
        \cA = \begin{pNiceArray}{c:c:c:c}[first-row,margin,code-for-first-row=\scriptstyle] 
            \begin{array}{*{3}{>{\scriptstyle}c}} \overline{1} \end{array} & \begin{array}{*{3}{>{\scriptstyle}c}} 2 & \cdots & m-1 \end{array} & m & \overline{m}\\
            0 & J_{1,m-2}  & 1 & 1\\ \hdashline
             & &  & \\
            J_{m-2,1} & J_{m-2}-I_{m-2}& v^T &J_{m-2,1} \\
             & & &
        \end{pNiceArray},
    $$
    where~$v$ is an arbitrary vector in~$\Z_2^{m-2}$.
    Indeed, none of the identities obtained above imposes any condition on~$v$.
    
    Now, assume that~$r > 0$.
    Substituting~$a_{\overline{t},t}=0$ into~(3), (4), and (5) in Table~\ref{tab:minor-cal-relations}, for each~$1 \leq s \leq r+1 < t \leq m-1$, we have~$a_{s,t}a_{\overline{t},\overline{s}} = a_{s,\overline{s}}a_{\overline{t},m} + a_{\overline{t},\overline{s}} = a_{s,t} = 1$.
    Thus,
    $$
        \cA = \begin{pNiceArray}{c:c:c:c}[first-row,margin,code-for-first-row=\scriptstyle]
            \begin{array}{*{3}{>{\scriptstyle}c}} \overline{1} & \cdots & \overline{r+1} \end{array} & \begin{array}{*{3}{>{\scriptstyle}c}} r+2 & \cdots & m-1 \end{array} & m & \overline{m}\\
            &&&\\
            \ast & J_{r+1,m-r-2}  & J_{r+1,1} & \ast \\
            &&&\\ \hdashline
            &&&\\
            J_{m-r-2,r+1} & J_{m-r-2}-I_{m-r-2}& \ast &J_{m-r-2,1} \\
            &&&
        \end{pNiceArray}. 
    $$
    Moreover, $a_{s,\overline{s}}a_{\overline{t},m}=0$ for each~$1 \leq s \leq r+1 < t \leq m-1$.
    Here the starred blocks denote entries that are not determined at this stage.
    Combining the results obtained so far with~\eqref{eq:minor-3b}, we get~$a_{\overline{t},m} = a_{s,\overline{s'}} = 1$ for all distinct~$1\leq s,s'\leq r+1$ and all~$r+2\leq t\leq m-1$.
    Using this in the relation~$a_{s,\overline{s}}a_{\overline{t},m}=0$, we get~$a_{s,\overline{s}}=0$ for all~$1\leq s\leq r+1$.
    Finally, substituting these relations into~(6) in Table~\ref{tab:minor-cal-relations} gives~$a_{s,\overline{m}} = 1$ for all~$1\leq s\leq r+1$.
    Therefore, $\cA$ coincides with~$A_{(1,\ldots,1)}$.
\end{proof}

\begin{lemma}\label{lemma:main2}
    Assume that
    \begin{enumerate}
        \item $\Lambda(\cA)$ satisfies the non-singularity condition,
        \item $a_{s,\overline{s}}=1$ for all~$1 \leq s \leq r+1$, and
        \item $a_{s,\overline{m}} = 0$ for some~$1 \leq s \leq r+1$.
    \end{enumerate}
    Then the following conclusions hold. 
    If~$r=0$, then~$\cA = B^1_v$ for some~$v\in \Z_2^{m-2}$. 
    If~$r>0$, then~$\cA = B^j_{\left(0,\ldots,0\right)}$ for some~$j\in \{1,\ldots,r+1 \}$.
\end{lemma}
\begin{proof}
Take~$s_0\in\{1,\ldots,r+1\}$ with~$a_{s_0,\overline{m}}=0$.
Applying the identities in Table~\ref{tab:minor-cal-relations} with $s=s_0$ yields the relations
\begin{equation}\label{eq:lemma2-1}
    \begin{cases}
        a_{s_0,m} = a_{s_0,t} = a_{\overline{t},\overline{m}} = 1, \text{ and } \\
        a_{\overline{t},m}+1 = a_{\overline{t},t}+1= a_{\overline{t},t'} = a_{\overline{t},\overline{s}} = a_{\overline{t},\overline{s_0}}
    \end{cases}
\end{equation}
for all~$1\leq s\leq r+1$ and distinct~$r+2\leq t,t'\leq m-1$.
The first condition in~\eqref{eq:lemma2-1} determines the row of~$\cA$ indexed by~$s_0$ as
$$
    \begin{pNiceArray}{c:c:cc}[first-row,margin,code-for-first-row=\scriptstyle]
        \begin{array}{*{6}{>{\scriptstyle}c}} \overline{1} & \cdots & \overline{r+1} \end{array} & \begin{array}{*{6}{>{\scriptstyle}c}} r+2 & \cdots & m-1 \end{array} & m & \overline{m}\\
        \ast & J_{1,m-r-2} & 1 & 0
    \end{pNiceArray},
$$
where the starred block denotes entries that are not yet determined, with the only exception that~$a_{s_0,\overline{s_0}}=1$.
Setting
$$
    v= (a_{\overline{r+2},\overline{s_0}},\ldots, a_{\overline{m-1},\overline{s_0}}) \in \Z_2^{m-r-2},
$$
the second set of relations in~\eqref{eq:lemma2-1} implies that, for each~$r+2\leq t\leq m-1$, the row of~$\cA$ indexed by~$\overline{t}$ is equal to the corresponding row of the matrix in~\eqref{r=0_case3}.
For~$r=0$, this exhausts all rows to be determined, immediately yielding~$\cA=B_v^1$.

Assume now that~$r>0$, and suppose, for the sake of contradiction, that~$a_{\overline{t_0},\overline{s_0}}=1$ for some index~$t_0 \in \{r+2,\ldots,m-1\}$.
The second condition in~\eqref{eq:lemma2-1} necessitates~$a_{\overline{t_0},t_0} = a_{\overline{t_0},m} = 0$.
In conjunction with (5) of Table~\ref{tab:minor-cal-relations}, this implies~$a_{s,t_0} = 1$ for all~$1 \leq s \leq r+1$.
Because~$r>0$, we may choose an index~$s'\neq s_0$, whereupon evaluating~\eqref{eq:minor-3b} at these indices yields~$1=0$, a contradiction.
Thus, $a_{\overline{t},\overline{s_0}}$ must vanish for all~$r+2\leq t\leq m-1$, implying that~$v$ is the zero vector.
As a result, the second block of relations in~\eqref{eq:lemma2-1} simplifies to~$a_{\overline{t},m} = a_{\overline{t},t}= 1$ and~$a_{\overline{t},t'} = a_{\overline{t},\overline{s}} = 0$ for any~$1\leq s\leq r+1$ and distinct~$r+2\leq t,t'\leq m-1$.

Next, fix~$s_1 \in\{1,\ldots,r+1\}\setminus\{s_0\}$ and distinct~$t,t'\in\{r+2,\ldots,m-1\}$.
Putting~$s=s_1$ and~$s'=s_0$ in~\eqref{eq:minor-3b}, and using~$a_{s_0,t}=1$, we obtain~$a_{s_1,t}=0$.  
Substituting this into (5) of Table~\ref{tab:minor-cal-relations} with~$s=s_1$, we get~$a_{s_1,\overline{m}}=1$. 
Thus,
\begin{equation}\label{eq:lemma2-2}
    a_{s_1,t}=0, \quad \text{ and } \quad a_{s_1,\overline{m}}=1.
\end{equation}
Now putting~$s=s_0$ and~$s'=s_1$ in~\eqref{eq:minor-3b}, and using~\eqref{eq:lemma2-2}, we further obtain~$a_{s_0,\overline{s_1}}=1$.
In conclusion, the row of~$\cA$ indexed by~$s_0$ coincides with the vector in~\eqref{eq:B_s0}.
When~$r\geq 2$, choose~$s_2\in\{1,\ldots,r+1\}\setminus\{s_0,s_1\}$.
Putting~$s=s_1$ and~$s'=s_2$ in~\eqref{eq:minor-3b}, and using~\eqref{eq:lemma2-2}, we obtain~$a_{s_1,\overline{s_2}}=0$.
Setting up the replacement pair~$(X,Y) = \big(\{s_1,s_0,\overline{t},\overline{t'}\}, \{t,t',\overline{s_1},\overline{s_0}\}\big)$, the previously derived entries reduce the corresponding minor to
$$
    [\cA]_{X,Y} = \det
        \begin{pmatrix}
            1 & 1 & 1 & 1 \\
            a_{s_1,\overline{s_0}} & 1 & 0 & 0 \\
            0 & 0 & 1 & 0 \\
            0 & 0 & 1 & 1
        \end{pmatrix}
    = 1+a_{s_1,\overline{s_0}}.
$$
By Proposition~\ref{prop:minor}, the identity~$[\cA]_{X,Y}=1$ implies~$a_{s_1,\overline{s_0}}=0$.  
Moreover, by Table~\ref{tab:minor-cal-relations}~(1), we have~$a_{s_1,m}=1$.
Hence the row of~$\cA$ indexed by~$s_1$ coincides with the row of~$B_{(0,\ldots,0)}^{s_0}$ indexed by~$s_1$.
Since~$s_1$ was arbitrary in~$\{1,\ldots,r+1\}\setminus\{s_0\}$, we conclude that~$\cA=B_{(0,\ldots,0)}^{s_0}$, as desired.
\end{proof}

\begin{lemma}\label{lemma:main3}
    Assume that 
    \begin{enumerate}
        \item $\Lambda(\cA)$ satisfies the non-singularity condition,
        \item $a_{i,\overline{i}}=a_{i,\overline{m}}=1$ for all $1\leq i\leq r+1$, and
        \item $a_{\overline{t},m}=0$ for some $r+2\leq t\leq m-1$.
    \end{enumerate}
    Fix such an index~$t$. 
    Then the following conclusions hold. 
    If~$r=0$, then~$\cA=C_v^t$ for some~$v\in \Z_2^{m-2}$. 
    If~$r>0$, then~$\cA=C_v^t$, where~$v=(v_{r+2},\ldots,v_{m-1})\in \Z_2^{m-r-2}$ is determined by~$v_t=1$ and~$v_i=0$ for all~$i\neq t$.
\end{lemma}
\begin{proof}
    Take~$t_0 \in \{r+2, \ldots, m-1\}$ such that~$a_{\overline{t_0}, m} = 0$, and~$t_1\in \{r+2,\ldots,m-1\}\setminus\{t_0\}$.
    Set~$v_t := a_{\overline{t},t_0}$ for~$r+2 \leq t \leq m-1$, and~$v := (v_{r+2},\ldots,v_{m-1})$.
    We apply the following relations from
    Table~\ref{tab:minor-cal-relations}:
    \begin{itemize}
        \item relation~(1);
        \item relations~(2), (3), (4), and~(5), with $t=t_0$ and $t=t_1$ in each relation; and
        \item relation~(6) with $(t,t')=(t_1,t_0)$.
    \end{itemize}
    Together with the assumptions of this lemma, these give, for each~$1\leq s\leq r+1$,
    $$
        \begin{cases}
            a_{s,\overline{s}}=a_{s,m}=a_{s,\overline{m}}=1, \quad a_{s,t_1}=0, \quad a_{s,t_0}= v_{t_0}+1; \\
            a_{\overline{t_1},\overline{s}}=0, \quad a_{\overline{t_1},t_1} = a_{\overline{t_1},m}= a_{\overline{t_1},\overline{m}}=1, \quad a_{\overline{t_1},t_0} = v_{t_1};\\
            a_{\overline{t_0},\overline{s}} = a_{\overline{t_0},t_1}= a_{\overline{t_0},\overline{m}}=1, \quad a_{\overline{t_0},m} = 0, \quad a_{\overline{t_0},t_0} = v_{t_0}.
        \end{cases}
    $$
    Moreover, if~$|\{r+2,\ldots,m-1\}| \geq 3$, then for~$t_2\in \{r+2,\ldots,m-1\}\setminus\{t_0,t_1\}$, applying (6) of Table~\ref{tab:minor-cal-relations} with $(t,t')=(t_1,t_2)$ gives~$a_{\overline{t_1},t_2}=0$.
    When~$r=0$, reading off the entries above gives~$\cA=C_v^{t_0}$.
    There is no further restriction on~$v$, so~$v$ is arbitrary.
    This completes the proof in the case~$r=0$.
    
    Assume now that~$r>0$.
    Evaluating \eqref{eq:minor-3b} at~$t=t_0$ and combining it with the prior identity~$a_{s,t_0} = v_{t_0}+ 1$, for distinct indices~$s,s'\in\{1,\ldots,r+1\}$, we have
    $$
        v_{t_0}=1,\quad a_{s,t_0}=0,\quad \text{ and } \quad a_{s,\overline{s'}}=0.
    $$
    Let~$t\in \{r+2,\ldots,m-1\}\setminus\{t_0\}$, and choose distinct~$s,s'\in \{1,\ldots,r+1\}$. 
    Consider the replacement pair~$(X,Y)=(\{s,s',\overline{t},\overline{t_0}\}, \{t,t_0,\overline{s},\overline{s'}\})$.
    Using the entries obtained above, the corresponding minor is
    $$
        [\cA]_{X,Y} = \det
            \begin{pmatrix}
                1 & 0 & 0 & 0 \\
                0 & 1 & 0 & 0 \\
                0 & 0 & 1 & v_t \\
                1 & 1 & 1 & 1
            \end{pmatrix}
        = 1 + a_{\overline{t},t_0}.
    $$
    By Proposition~\ref{prop:minor}, the identity~$[\cA]_{X,Y} = 1$ implies that~$v_t=0$ for all~$t\neq t_0$, as desired.
\end{proof}

We now prove Theorem~\ref{thm:r=0_sc}.
\begin{proof}[Proof of Theorem~\ref{thm:r=0_sc}]
    It is straightforward to check that~$\Lambda(A_0)$ is a characteristic matrix of~$\Bier\Delta^{m-1}_r$ for~$0\leq r\leq \left\lfloor (m-3)/2\right\rfloor$.
    By Remark~\ref{remark:main}, for each matrix~$\cA$ given in cases~(2), (3), and~(4) of this theorem, the matrix~$\Lambda(\cA)$ is a characteristic matrix.
    
    It remains to show that no other cases occur. 
    By Lemmas~\ref{lemma:main1}, \ref{lemma:main2}, and~\ref{lemma:main3}, all possibilities fall under cases~(2), (3), and~(4), except when~$a_{s,\overline{s}}=a_{s,\overline{m}}=1$ for all~$1\leq s\leq r+1$ and~$a_{\overline{t},m}=1$ for all~$r+2\leq t\leq m-1$.
    In this remaining case, one checks directly that~$\cA=A_0$.
\end{proof}

\section{Topology of small covers over~$\Bier \Delta_{r}^{m-1}$}\label{sec4}

For a characteristic matrix~$\Lambda$ over~$\Bier\Delta^{m-1}_{r}$, we denote by
$$
    M(\Bier\Delta^{m-1}_{r},\Lambda)
$$
the small cover associated with the pair~$(\Bier\Delta^{m-1}_{r},\Lambda)$.
In this section, we investigate the topology of~$M(\Bier\Delta^{m-1}_{r},\Lambda)$, starting with the case~$r \in \{0,m-3\}$.

We first recall the connected sum construction for simplicial spheres.
Let~$\Gamma_1$ and~$\Gamma_2$ be~$(n-1)$-dimensional simplicial spheres, and let~$\sigma$ be a facet of both~$\Gamma_1$ and~$\Gamma_2$.
The connected sum~$\Gamma_1\#_{\sigma}\Gamma_2$ is obtained from~$\Gamma_1\cup\Gamma_2$ by identifying the two copies of~$\sigma$ and then deleting the common facet.

For closed~$n$-manifolds~$M_1$ and~$M_2$, $M_1 \# M_2$ denotes their connected sum.
We also write~$\#^k M$ for the connected sum of~$k$ copies of a closed~$n$-manifold~$M$.
\begin{proposition}[\cite{Davis-Januszkiewicz1991}] \label{eqn:connected_sum}
    Let~$\Lambda$ be a characteristic matrix over a simplicial complex~$\Gamma_1 \#_{\sigma} \Gamma_2$.
    Let~$\Lambda_i$ denote the submatrix of~$\Lambda$ induced by the vertex set of~$\Gamma_i$ for~$i=1,2$.
    If the submatrix of~$\Lambda$ with columns indexed by~$\sigma$ has determinant~$1$, then~$\Lambda_i$ satisfies the non-singularity condition for~$i=1,2$.
    Moreover, there is a homeomorphism
    $$
        M(\Gamma_1\#_{\sigma}\Gamma_2,\Lambda) \cong M(\Gamma_1,\Lambda_1)\# M(\Gamma_2,\Lambda_2).
    $$
\end{proposition}

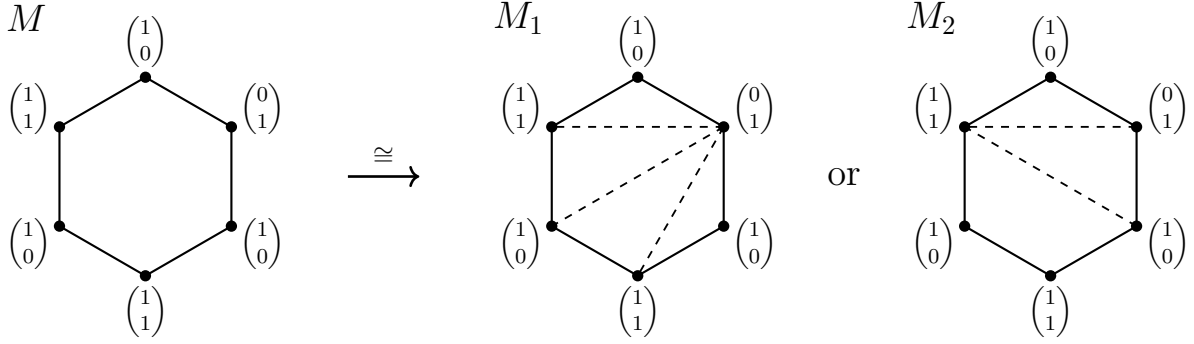
\begin{figure}
    \centering
    \begin{tikzpicture}[
      scale=1.05,
      every node/.style={font=\small},
      vtx/.style={circle, fill=black, inner sep=1.5pt},
      edge/.style={line width=0.8pt},
      glue/.style={dashed, line width=0.7pt}
    ]
    
    \begin{scope}[shift={(-3.1,0)}]
      \node at (-1.5,2.0) {\Large{$M$}};
    
      \foreach \i/\lab in {
        90/{\scriptsize$\begin{pmatrix}1\\0\end{pmatrix}$},
        30/{\scriptsize$\begin{pmatrix}0\\1\end{pmatrix}$},
        -30/{\scriptsize$\begin{pmatrix}1\\0\end{pmatrix}$},
        -90/{\scriptsize$\begin{pmatrix}1\\1\end{pmatrix}$},
        -150/{\scriptsize$\begin{pmatrix}1\\0\end{pmatrix}$},
        150/{\scriptsize$\begin{pmatrix}1\\1\end{pmatrix}$}
      }{
        \coordinate (L\i) at (\i:1.25);
        \node[vtx] at (L\i) {};
        \node at (\i:1.7) {\lab};
      }
    
      \draw[edge]
        (L90) -- (L30) -- (L-30) -- (L-90) -- (L-150) -- (L150) -- cycle;
    \end{scope}

    \draw[->, line width=1.0pt] (2.55-3.1,0) -- (3.45-3.1,0) node[midway, above] {$\cong$};

    \begin{scope}[shift={(3.1,0)}]
    \node at (-1.5,2.0) {\Large{$M_1$}};
    
      \foreach \i/\lab in {
        90/{\scriptsize$\begin{pmatrix}1\\0\end{pmatrix}$},
        30/{\scriptsize$\begin{pmatrix}0\\1\end{pmatrix}$},
        -30/{\scriptsize$\begin{pmatrix}1\\0\end{pmatrix}$},
        -90/{\scriptsize$\begin{pmatrix}1\\1\end{pmatrix}$},
        -150/{\scriptsize$\begin{pmatrix}1\\0\end{pmatrix}$},
        150/{\scriptsize$\begin{pmatrix}1\\1\end{pmatrix}$}
      }{
        \coordinate (L\i) at (\i:1.25);
        \node[vtx] at (L\i) {};
        \node at (\i:1.7) {\lab};
      }
    
      \draw[edge]
        (L90) -- (L30) -- (L-30) -- (L-90) -- (L-150) -- (L150) -- cycle;
    
      \draw[glue] (L30)--(L150);
      \draw[glue] (L30)--(L-90);
      \draw[glue] (L30)--(L-150);
    \end{scope}
    \node at (11.4/2,0) {\Large{or}};
    \begin{scope}[shift={(9.3-1,0)}]
    \node at (-1.5,2.0) {\Large{$M_2$}};
    
      \foreach \i/\lab in {
        90/{\scriptsize$\begin{pmatrix}1\\0\end{pmatrix}$},
        30/{\scriptsize$\begin{pmatrix}0\\1\end{pmatrix}$},
        -30/{\scriptsize$\begin{pmatrix}1\\0\end{pmatrix}$},
        -90/{\scriptsize$\begin{pmatrix}1\\1\end{pmatrix}$},
        -150/{\scriptsize$\begin{pmatrix}1\\0\end{pmatrix}$},
        150/{\scriptsize$\begin{pmatrix}1\\1\end{pmatrix}$}
      }{
        \coordinate (L\i) at (\i:1.25);
        \node[vtx] at (L\i) {};
        \node at (\i:1.7) {\lab};
      }
    
      \draw[edge]
        (L90) -- (L30) -- (L-30) -- (L-90) -- (L-150) -- (L150) -- cycle;
    
      \draw[glue] (L30)--(L150);
      \draw[glue] (L-30)--(L150);
    
    \end{scope}
    
    \end{tikzpicture} 
    \caption{Connected sum decompositions over a hexagon} \label{fig:bier-delta0-m3-connected-sum}
\end{figure}
For example, let~$M=M(\Bier\Delta^2_0,\Lambda)$ be the small cover over the hexagonal Bier sphere~$\Bier\Delta^2_0$ shown in Figure~\ref{fig:bier-delta0-m3-connected-sum}.
The dotted lines mark the missing faces used to decompose~$M$ successively as connected sums.
The two decompositions shown in the figure yield the connected sum descriptions~$M_1$ and~$M_2$.
Then
$$
    M_1\cong \#^{4}\RP^{2} \quad \text{ and } \quad M_2\cong \RP^{2}\#\RP^{2}\#(\RP^{1}\times \RP^{1}).
$$
Note that the two descriptions yield the same surface up to homeomorphism, since
$$
    \#^{3}\RP^{2} \cong  \RP^{2}\#(\RP^{1}\times \RP^{1}).
$$

Assume that~$m\geq 4$. 
Let~$\Gamma=\partial\Delta^{m-1}\#_{\sigma}\partial\Delta^{m-1}$ be the connected sum of two boundary complexes of the~$(m-1)$-simplex along a common facet~$\sigma=\{p_1,\ldots,p_{m-1}\}$.
Then the vertex set of~$\Gamma$ is $\{p_1,\ldots,p_{m-1},u,v\}$, where~$u$ and~$v$ are the two vertices not contained in~$\sigma$, one from each copy of~$\partial\Delta^{m-1}$.
Let~$\Lambda$ be a mod~$2$ characteristic matrix for~$\Gamma$ which has the form
\begin{equation}\label{eq:beforcoro5-1}
    \Lambda=\begin{pNiceArray}{c:cc}[first-row,margin,code-for-first-row=\scriptstyle]
        \begin{array}{*{4}{>{\scriptstyle}c}} p_1 & \cdots & p_{m-2} & v \end{array}  & p_{m-1} & u\\
        &a_1 &\ b_1\\
        I_{m-1} & \vdots & \vdots\\
        & a_{m-1} & b_{m-1}
    \end{pNiceArray}.
\end{equation}
In~\cite{Choi-Masuda-Suh2010Quasitoric}, the small cover~$M(\Gamma,\Lambda)$ is called a~$2$-stage generalized real Bott manifold, and its homeomorphism type is classified as follows.
\begin{equation}\label{eq:beforcoro5-2}
    M(\Gamma,\Lambda) \cong
        \begin{cases}
            \RP^{m-1}\#\RP^{m-1},& \text{if }a_{m-1}=1,\\
            \RP\bigl(\gamma\oplus\R^{m-2}\bigr),& \text{if }a_{m-1}=0 \text{ and $m \cdot (b_1 +\cdots + b_{m-1}) \equiv 1 \pmod 2$},\\
            \RP^{m-2}\times\RP^1,& \text{otherwise},
        \end{cases}
\end{equation}
where~$\gamma$ denotes the tautological real line bundle over~$\RP^1$.

To state the classification of small covers over~$\Bier\Delta_{0}^{m-1}$, we recall the parametrization of characteristic matrices given in Theorem~\ref{thm:r=0_sc}.

\begin{corollary}\label{thm:connected_sum}
    For~$m\geq 4$ and~$r\in\{0,m-3\}$, every small cover over~$\Bier\Delta_{r}^{m-1}$ is homeomorphic to one of the following manifolds
    $$
        \#^{m+1}\RP^{m-1}, \quad (\#^{m-1}\RP^{m-1})\#(\RP^{m-2}\times \RP^1), \quad \text{or} \quad (\#^{m-1}\RP^{m-1})\#\RP(\gamma\oplus \R^{m-2}),
    $$
    where~$\gamma$ denotes the tautological line bundle over~$\RP^1$.
    In particular, the last case occurs only when~$m$ is odd.
\end{corollary}
\begin{proof}
    Since~$\Bier\Delta_{0}^{m-1}$ and~$\Bier\Delta_{m-3}^{m-1}$ are combinatorially isomorphic, it is enough to consider the case~$r=0$.
    We use the following notation.
    \begin{itemize}
        \item $\Gamma_0$ denotes the boundary of the $(m-1)$-simplex with vertex set $[\overline{m}]$;
        \item for $1\leq i\leq m$, let
            $$
                \sigma_i:=[\overline{m}]\setminus\{\overline{i}\},
            $$
            and let $\Gamma_i$ denote the boundary of the $(m-1)$-simplex with vertex set $V_i :=\sigma_i\cup\{i\}$.
    \end{itemize}
    Then~$\Bier\Delta_0^{m-1}$ admits a connected sum decomposition obtained from~$\Gamma_0$ by taking connected sums with~$\Gamma_i$ along the common facet~$\sigma_i$ for~$1\leq i\leq m$. 
    The resulting complex does not depend on the order in which these connected sums are taken.
    
    Recall the notation~$\Lambda(\cA)$ from~\eqref{eq:mat_form}, and put~$\Lambda=\Lambda(\cA)$.
    For a vertex subset~$W\subset [m]\cup[\overline{m}]$, write~$\Lambda|_W$ for the submatrix of~$\Lambda$ obtained by restricting to the columns indexed by~$W$.
    
    Assume first that~$\cA=A_0$.
    Then~$\det\Lambda|_{\sigma_i}=1$ for all~$1\leq i\leq m$.
    The above connected sum decomposition of~$\Bier\Delta_0^{m-1}$, together with repeated applications of Proposition~\ref{eqn:connected_sum}, gives a homeomorphism
    \begin{align*}
        M(\Bier\Delta_{0}^{m-1},\Lambda) &\cong  M(\Gamma_0,\Lambda|_{[\overline{m}]}) \# M(\Gamma_1,\Lambda|_{V_1}) \# \cdots \# M(\Gamma_m,\Lambda|_{V_m}) \\
        & \cong \#^{m+1}\RP^{m-1}.
    \end{align*}
    The second homeomorphism follows from the fact that the unique small cover over each~$\Gamma_i$ is~$\RP^{m-1}$, since each~$\Gamma_i$ is the boundary of an~$(m-1)$-simplex.
    
    Now assume that~$\cA$ is one of the matrices in Definition~\ref{def:mat}.
    By Theorem~\ref{thm:r=0_sc}, $\Lambda$ is a characteristic matrix over~$\Bier\Delta_{0}^{m-1}$.
    Define~$\nu \in[m]$ by
    $$
        \nu = \begin{cases}
            m, & \text{if } \cA = A_v ,\\
            1, & \text{if } \cA = B_v^1, \\
            t, & \text{if } \cA = C_v^t,
        \end{cases}
    $$
    where~$2 \leq t \leq m-1$ and~$v=(v_2,\ldots, v_{m-1})\in \Z_2^{m-2}$.
    Then, $\det \Lambda|_{\sigma_i} =1$ for all~$1 \leq i \leq m$ except for~$i = \nu$.
    Applying Proposition~\ref{eqn:connected_sum} repeatedly to the connected sum decomposition of~$\Bier\Delta_0^{m-1}$, while leaving the summand~$\Gamma_\nu$ attached to~$\Gamma_0$, gives a homeomorphism
    $$
        M(\Bier\Delta_{0}^{m-1},\Lambda) \cong \big( \#^{m-1}\RP^{m-1}\big) \# \big(M(\Gamma_0 \#_{\sigma_\nu}\Gamma_{\nu},\Lambda|_{\{\nu\}\cup[\overline{m}]})\big). 
    $$
    Since~$\Gamma_0 \#_{\sigma_\nu}\Gamma_{\nu}$ is simplicially isomorphic to~$\partial\Delta^{m-1}\#_{\sigma}\partial\Delta^{m-1}$ for a common facet~$\sigma$, the homeomorphism type of the latter connected summand is determined by~\eqref{eq:beforcoro5-2}.
    In the notation of~\eqref{eq:beforcoro5-1}, a direct calculation gives~$a_{m-1} =0$ and   
    $$
        b_1+\cdots +b_{m-1} =
            \begin{cases}
                \sum\limits_{i=2}^{m-1}(1+v_i) & \text{for } \cA = A_v, \\
                \sum\limits_{i=2}^{m-1}v_i & \text{for } \cA = B^1_v, \\
                1+\sum\limits_{i=2}^{m-1}v_i & \text{for } \cA = C_v^t.
            \end{cases}
    $$
    Since the~$\Z_2$-vector~$v$ can be chosen freely, the sum~$b_1+\cdots+b_{m-1}$ can attain both values~$0$ and~$1$ in~$\Z_2$.
    This completes the proof.
\end{proof}

We now consider the case~$r>0$.
Our goal is to compute the rational Betti numbers of small covers over~$\Bier\Delta^{m-1}_{r}$.
Let~$\Gamma$ be a simplicial sphere on a finite vertex set~$S$, and let~$\Lambda$ be a characteristic matrix over~$\Gamma$. 
We denote by~$\row\Lambda$ the row space of~$\Lambda$.
For each~$\omega\in\row\Lambda$, we use the following notation.
\begin{itemize}
    \item $V_\omega\subseteq S$ denotes the set of vertices indexing the nonzero entries of~$\omega$, and
    \item $\Gamma_\omega$ denotes the full subcomplex of~$\Gamma$ induced by~$V_\omega$. 
\end{itemize}

For convenience, we write~$\beta_i(X)$ for the~$i$th rational Betti number of a topological space~$X$.
\begin{lemma}[\cite{Choi-Park2017Cohomology}]\label{lem:choi-park2017cohomology}
    Let~$M$ be the small cover associated with~$(\Gamma,\Lambda)$.
    Then the~$k$th rational Betti number of~$M$ is given by
    $$
        \beta_k(M)=\sum_{\omega\in \row \Lambda}\widetilde{\beta}_{k-1}(\Gamma_\omega),
    $$
    where~$\widetilde{\beta}_{k-1}$ denotes the~$(k-1)$th reduced Betti number.
\end{lemma}

Thus, for a characteristic matrix~$\Lambda$ over~$\Bier\Delta_r^{m-1}$, the rational Betti numbers of~$M(\Bier\Delta^{m-1}_{r},\Lambda)$ are obtained from the reduced Betti numbers of the full subcomplexes~$(\Bier\Delta^{m-1}_{r})_{\omega}$ for all~$\omega\in\row\Lambda$.

For a simplicial complex~$K$ on~$[m]$, the vertex set of~$\Bier K$ is~$[m]\cup[\overline{m}]$. 
Hence every vertex subset of~$\Bier K$ is uniquely of the form~$I\cup\overline{J}$ for some~$I,J\subset[m]$.
We write~$\Lk(\sigma,\Gamma)$ for the link of~$\sigma$ in~$\Gamma$ and~$\Sigma^k\Gamma$ for the~$k$-fold suspension of~$\Gamma$, with the convention~$\Sigma^0\Gamma=\Gamma$.
Lemma~\ref{thm:full_subcomplex_Bier} gives the homotopy types needed for this computation.

\begin{lemma}[\cite{Choi-Yoon-Yu2026}]\label{thm:full_subcomplex_Bier}
    Let $K$ be a simplicial complex on $[m]$, and let $I,J$ be subsets of $[m]$.
    Then the full subcomplex $(\Bier K)_{I\cup\overline{J}}$ of $\Bier K$ with respect to $I\cup\overline{J}$ is homotopy equivalent to
    \begin{enumerate}
        \item\label{thm:full_subcomplex_1} \makebox[4cm][l]{the $(|I|-1)$-sphere,} if $I=J$, $I\in K$, and $\overline{I}\in\widehat{K}$,
        \item\label{thm:full_subcomplex_2} \makebox[4cm][l]{the $(|I|-2)$-sphere,} if $I=J$, $I\notin K$, and $\overline{I}\notin\widehat{K}$,
        \item\label{thm:full_subcomplex_3} \makebox[4cm][l]{$\Sigma^{|J|}\Lk(J,K\vert_I)$, } if $J\subsetneq I$ and $J\in K$,
        \item\label{thm:full_subcomplex_4} \makebox[4cm][l]{$\Sigma^{|I|}\Lk(\overline{I},\widehat{K}\vert_{\overline{J}})$,} if $J\supsetneq I$ and $\overline{I}\in\widehat{K}$, and
        \item\label{thm:full_subcomplex_5} \makebox[4cm][l]{a point,} otherwise.
    \end{enumerate}
\end{lemma}

Combining Lemmas~\ref{lem:choi-park2017cohomology} and~\ref{thm:full_subcomplex_Bier}, one can compute the rational Betti numbers of all small covers over~$\Bier\Delta_r^{m-1}$.
Note that~$\Bier\Delta_r^{m-1} \cong\Bier\Delta_{m-r-3}^{m-1}$, so it suffices to consider
$$
0< r\leq \left\lfloor\frac{m-3}{2}\right\rfloor.
$$
Among them, the case corresponding to~$A_0$ was already computed in~\cite{Choi-Yoon-Yu2026}.
To recall this result in the following proposition, we introduce the following notation for indicator functions on~$\Z$.
For a subset~$A\subseteq \Z$, let~$\ind_A\colon \Z \to \{0,1\}$ denote the indicator function of~$A$, defined by
$$
    \ind_A(x)= 
        \begin{cases}
            1, & \text{if } x \in A,\\
            0, & \text{if } x \notin A.
        \end{cases}
$$

\begin{proposition}[\cite{Choi-Yoon-Yu2026}]\label{prop:Bettisk}
Let~$m\geq 5$ and~$0< r \leq \left\lfloor\frac{m-3}{2}\right\rfloor$.
The rational Betti numbers of $M:=M(\Bier\Delta^{m-1}_r,\Lambda(A_0))$ are given by, for each~$k \geq 0$,
$$
    \beta_{2k}(M) = \dbinom{m}{2k} \cdot \ind_{\{0,1,\ldots,r+1\}}(2k),
$$
and, for each~$k \geq 1$,
$$
    \beta_{2k-1}(M) = \dbinom{m}{2k} \cdot \ind_{\{m-r-1,\ldots,m\}}(2k).
$$
\end{proposition}

We now compute the rational Betti numbers for the remaining characteristic matrices described in Theorem~\ref{thm:r=0_sc}.
\begin{corollary}\label{thm:betti_r_positive}
    Let~$m\geq5$ and~$0<r\leq\left\lfloor\frac{m-3}{2}\right\rfloor$.
    Let~$\cA$ be one of the mod~$2$ matrices described in Theorem~\ref{thm:r=0_sc}, except for~$A_0$.
    Set~$M:=M(\Bier\Delta_r^{m-1},\Lambda(\cA))$.
    Then, for~$k\geq0$,
    $$
        \beta_{2k}(M) = \dbinom{m-1}{2k} \cdot \ind_{\{0,1,\ldots,r+1\}}(2k),
    $$
    and, for~$k\geq1$,
    $$
        \beta_{2k-1}(M) = \dbinom{m-1}{2k} \cdot \ind_{\{m-r-1,\ldots,m-1\}}(2k) + \dbinom{m-2}{r+1} \cdot \ind_{\{m-r-1\}}(2k) + \dbinom{m-2}{r} \cdot \ind_{\{r+2\}}(2k).
    $$
\end{corollary}
\begin{proof}
Let~$1\leq s\leq r+1$ and~$r+2\leq t\leq m-1$, and let~$\chi^t = (\chi^t_{r+2},\ldots,\chi^t_{m-1})\in\Z_2^{m-r-2}$ be defined by
$$
    \chi^t_i= \ind_{\{t\}}(i)
$$
for each~$r+2\leq i\leq m-1$.
For~$\cA\in \{ A_{(1,\ldots,1)}, B^s_{(0,\ldots,0)}, C^t_{\chi^t} \}$, we define~$\nu(\cA)\in[m]$ by
$$
    \nu(\cA)=
    \begin{cases}
        m, & \text{if } \cA = A_{(1,\ldots,1)},\\
        s, & \text{if } \cA=B^s_{(0,\ldots,0)},\\
        t, & \text{if } \cA=C^t_{\chi^t}.
    \end{cases}
$$
For each such~$\cA$, let~$v_i$ denote the $i$th row of~$\Lambda(\cA)$ for each~$i\in[m-1]$.
Fix an arbitrary element~$\omega\in\row\Lambda(\cA)$ and write
$$
    \omega=v_{i_1}+\cdots+v_{i_N},
$$
where~$i_1,\ldots,i_N$ are pairwise distinct elements of~$[m-1]$.

\noindent \underline{\textbf{CASE 1. $\cA = A_{(1,\ldots,1)}$:}}
When~$N$ is even, we have~$V_\omega=I\cup\overline{I}$, where
$$ 
    I=\{i_1,\ldots,i_N\}\subseteq[m]\setminus\{\nu(\cA)\} 
$$
and~$|I|=N$ is even.
When~$N$ is odd, $V_\omega=I\cup\overline{J}$ for some~$I,J\subseteq[m]$ such that each~$i_\ell$ belongs to exactly one of~$I$ and~$J$.
Moreover,~$\nu(\cA)=m\in I\cap J$.

\noindent \underline{\textbf{CASE 2. $\cA = B^s_{(0,\ldots,0)}$ or $\cA = C^t_{\chi^t}$:}}
If~$\nu(\cA) \notin \{i_1,\ldots,i_N\}$, then~$V_\omega = I \cup \overline{I}$ for some~$I\subseteq[m]\setminus\{\nu(\cA)\}$ such that~$|I|$ is even.
Otherwise,~$V_\omega=I\cup\overline{J}$ for some~$I,J\subseteq[m]$ such that~$\nu(\cA)\in I\cap J$ and~$m \in (I \cup J) \setminus (I \cap J)$.

Combining the two cases, for each such~$\cA$, Lemma~\ref{thm:full_subcomplex_Bier} shows that~$(\Bier\Delta_r^{m-1})_\omega$ has a nonzero reduced Betti number exactly in the following three cases.
\begin{itemize}
    \item $V_\omega=I\cup\overline{I}$ for some $I\subseteq[m]\setminus\{\nu(\cA)\}$ with~$|I|$ even;
    \item $V_\omega=[m]\cup\{\overline{\nu(\cA)}\}$, which is possible precisely when~$r$ is even;
    \item $V_\omega=\{\nu(\cA)\}\cup[\overline{m}]$, which is possible precisely when~$m-r$ is odd. 
\end{itemize}
We further use Lemma~\ref{thm:full_subcomplex_Bier} to compute the full subcomplexes corresponding to these three cases. 
For~$I\subseteq[m]\setminus\{\nu(\cA)\}$ with~$|I|=2k$, 
$$
    (\Bier\Delta_r^{m-1})_{I\cup\overline{I}} \simeq 
    \begin{cases} 
        S^{2k-1}, & \text{if } 2k\leq r+1,\\ 
        S^{2k-2}, & \text{if } 2k>m-r-2. 
    \end{cases}
$$
For each~$2k$, there are exactly~$\binom{m-1}{2k}$ such subsets~$I$. 
For the other two cases, we have
$$ 
    (\Bier\Delta_r^{m-1})_{[m]\cup\{\overline{\nu(\cA)}\}} \simeq \Sigma\Delta_{r-1}^{m-2} 
    \quad \text{and} \quad 
    (\Bier\Delta_r^{m-1})_{\{\nu(\cA)\}\cup[\overline{m}]} \simeq \Sigma\Delta_{m-r-4}^{m-2}
$$
Since the skeleton of a simplex is shellable, it is homotopy equivalent to a wedge of spheres of its top dimension~\cite[Theorems~3.1.3 and~3.1.7]{Wachs2007}.
Comparing reduced Euler characteristics, we obtain
$$
    \Sigma\Delta_{r-1}^{m-2}
    \simeq
    \bigvee^{\binom{m-2}{r}} S^r
    \quad\text{and}\quad
    \Sigma\Delta_{m-r-4}^{m-2}
    \simeq
    \bigvee^{\binom{m-2}{r+1}} S^{m-r-3}.
$$

The preceding computations give the following values for
$$
    \sum_{\omega\in\row\Lambda}
    \widetilde{\beta}_{i}\bigl((\Bier\Delta_r^{m-1})_\omega\bigr).
$$
For~$i=2k-1$, this sum is given by
$$
    \binom{m-1}{2k}
    \cdot\ind_{\{0,1,\ldots,r+1\}}(2k).
$$
For~$i=2k-2$, it is given by
\begin{align*}
   & \binom{m-1}{2k}
    \cdot\ind_{\{m-r-1, \ldots, m-1\}}(2k)+
    \binom{m-2}{r+1}
    \cdot\ind_{\{m-r-3\}}(2k-2)
    +
    \binom{m-2}{r}
    \cdot\ind_{\{r\}}(2k-2) \\
  = & \binom{m-1}{2k}
    \cdot\ind_{\{m-r-1, \ldots, m-1\}}(2k)+
    \binom{m-2}{r+1}
    \cdot\ind_{\{m-r-1\}}(2k)
    +
    \binom{m-2}{r}
    \cdot\ind_{\{r+2\}}(2k).
\end{align*}
By Lemma~\ref{lem:choi-park2017cohomology}, the corollary follows.
\end{proof}

\providecommand{\bysame}{\leavevmode\hbox to3em{\hrulefill}\thinspace}
\providecommand{\MR}{\relax\ifhmode\unskip\space\fi MR }
\providecommand{\MRhref}[2]{%
  \href{http://www.ams.org/mathscinet-getitem?mr=#1}{#2}
}
\providecommand{\href}[2]{#2}


\begin{thebibliography}{10}

\bibitem{Bier1992}
Thomas Bier, \emph{A remark on {A}lexander duality and the disjunct join},
  preprint (1992).

\bibitem{Bjorner-Paffenholz-Sjostrand-Ziegler2005}
Anders Bj\"orner, Andreas Paffenholz, Jonas Sj\"ostrand, and G\"unter~M.
  Ziegler, \emph{Bier spheres and posets}, Discrete Comput. Geom. \textbf{34}
  (2005), no.~1, 71--86. \MR{2140883}

\bibitem{Cai-Chen-Lu2007}
Mingzhong Cai, Xin Chen, and Zhi L{\"u}, \emph{Small covers over prisms},
  Topology Appl. \textbf{154} (2007), no.~11, 2228--2234. \MR{2328006}

\bibitem{Choi2008}
Suyoung Choi, \emph{The number of small covers over cubes}, Algebr. Geom.
  Topol. \textbf{8} (2008), no.~4, 2391--2399. \MR{2465745}

\bibitem{Choi2010}
\bysame, \emph{The number of orientable small covers over cubes}, Proc. Japan
  Acad. Ser. A Math. Sci. \textbf{86} (2010), no.~6, 97--100. \MR{2680832}

\bibitem{Choi-Kaji-Theriault2017}
Suyoung Choi, Shizuo Kaji, and Stephen Theriault, \emph{Homotopy decomposition
  of a suspended real toric space}, Bol. Soc. Mat. Mex. (3) \textbf{23} (2017),
  no.~1, 153--161. \MR{3633130}

\bibitem{Choi-Masuda-Suh2010Quasitoric}
Suyoung Choi, Mikiya Masuda, and Dong~Youp Suh, \emph{Quasitoric manifolds over
  a product of simplices}, Osaka J. Math. \textbf{47} (2010), no.~1, 109--129.
  \MR{2666127}

\bibitem{Choi-Park2017Cohomology}
Suyoung Choi and Hanchul Park, \emph{On the cohomology and their torsion of
  real toric objects}, Forum Math. \textbf{29} (2017), no.~3, 543--553.
  \MR{3641664}

\bibitem{Choi-Park2019SmallCovers}
\bysame, \emph{Small covers over wedges of polygons}, J. Math. Soc. Japan
  \textbf{71} (2019), no.~3, 739--764. \MR{3984241}

\bibitem{Choi-Yoon-Yu2026}
Suyoung Choi, Younghan Yoon, and Seonghyeon Yu, \emph{Full subcomplexes of
  {B}ier spheres}, 2025, preprint, arXiv:2503.05385.

\bibitem{Davis-Januszkiewicz1991}
Michael~W. Davis and Tadeusz Januszkiewicz, \emph{Convex polytopes, {C}oxeter
  orbifolds and torus actions}, Duke Math. J. \textbf{62} (1991), no.~2,
  417--451. \MR{1104531 (92i:52012)}

\bibitem{Garrison-Scott2003}
Anne Garrison and Richard Scott, \emph{Small covers of the dodecahedron and the
  120-cell}, Proc. Amer. Math. Soc. \textbf{131} (2003), no.~3, 963--971.
  \MR{1937435}

\bibitem{Hasui2015}
Sho Hasui, \emph{On the classification of quasitoric manifolds over dual cyclic
  polytopes}, Algebr. Geom. Topol. \textbf{15} (2015), no.~3, 1387--1437.
  \MR{3361140}

\bibitem{Heudtlass-Katthan2012}
Inga Heudtlass and Lukas Katth\"an, \emph{Algebraic properties of {B}ier
  spheres}, Matematiche (Catania) \textbf{67} (2012), no.~1, 91--101.
  \MR{2927822}

\bibitem{Jevtic-Timotijevic-Zivaljevic2021}
Filip~D. Jevti\'c, Marinko Timotijevi\'c, and Rade~T. \v{Z}ivaljevi\'c,
  \emph{Polytopal {B}ier spheres and {K}antorovich-{R}ubinstein polytopes of
  weighted cycles}, Discrete Comput. Geom. \textbf{65} (2021), no.~4,
  1275--1286. \MR{4249903}

\bibitem{Limonchenko-Sergeev2024}
I.~Yu. Limonchenko and M.~A. Sergeev, \emph{Bier spheres and toric topology},
  Proc. Steklov Inst. Math. \textbf{326} (2024), 252--268. \MR{4855361}

\bibitem{Ivan-Marinko-Rade2025}
Ivan Limonchenko, Marinko Timotijevi{\'c}, and Rade {\v{Z}}ivaljevi{\'c},
  \emph{On a class of toric manifolds arising from simplicial complexes},
  (2025), preprint arXiv:2506.13547.

\bibitem{Limonchenko-Vavpeti2024}
Ivan Limonchenko and Aleš Vavpetič, \emph{Chromatic numbers, {B}uchstaber
  numbers and chordality of {B}ier spheres}, 2024, arXiv preprint
  arXiv:2412.20861.

\bibitem{Matousek2003book}
Ji\v{r}\'i Matou\v{s}ek, \emph{Using the {B}orsuk-{U}lam theorem},
  Universitext, Springer-Verlag, Berlin, 2003, Lectures on topological methods
  in combinatorics and geometry, Written in cooperation with Anders Bj\"orner
  and G\"unter M. Ziegler. \MR{1988723}

\bibitem{Murai2011}
Satoshi Murai, \emph{Spheres arising from multicomplexes}, J. Combin. Theory
  Ser. A \textbf{118} (2011), no.~8, 2167--2184. \MR{2834171}

\bibitem{Wachs2007}
Michelle~L. Wachs, \emph{Poset topology: tools and applications}, Geometric
  combinatorics, IAS/Park City Math. Ser., vol.~13, Amer. Math. Soc.,
  Providence, RI, 2007, pp.~497--615. \MR{2383132}

\bibitem{Wang-Chen2012}
Yanying Wang and Yanchang Chen, \emph{Small covers over products of a polygon
  with a simplex}, Turkish J. Math. \textbf{36} (2012), no.~1, 161--172.
  \MR{2881646}

\end{thebibliography}
\end{document}